\documentclass[12pt]{amsart}
\usepackage{amsfonts}
\usepackage{amssymb}

\usepackage{amsmath}
\usepackage{epsf}
\usepackage{epsfig}
\usepackage{colordvi}
\usepackage{fancybox}

\usepackage{xcolor}

\newcommand{\R}{\mathbb R}

\newtheorem{thm}{Theorem}[section]

\newtheorem{lemma}[thm]{Lemma}
\newtheorem{prop}[thm]{Proposition}
\theoremstyle{definition}
\newtheorem{defn}[thm]{Definition}

\theoremstyle{remark}
\newtheorem{remark}{Remark}[section]

\newtheorem{example}{Example}[section]

\newcommand{\ds}{\displaystyle}

\begin{document}

\title[Basic asymptotic invariants for time-like surfaces in $\mathbb R^3_1$  ]
{Basic asymptotic invariants for time-like surfaces \\ with real asymptotic lines in $\mathbb R^3_1$}%

\author{Ognian Kassabov}%
\address{Institute of Mathematics and Informatics, Bulgarian Academy of Sciences, \\ Acad. G. Bonchev Str. bl. 8, 1113, Sofia, Bulgaria}
\email{okassabov@abv.bg}


\begin{abstract}
The geometrically defined wide class of time-like surfaces in $\mathbb R^3_1$ admitting real asymptotic lines is considered.
A fundamental theorem of Bonnet-type is obtained for these surfaces. It states that a surface in this class is determined
(up to a motion) by four invariant functions, satisfying some
natural PDEs. Then canonical parameters are defined for these surfaces and it is proved that such a 
surface  is determined (up to a motion) in canonical parameters with only two invariant functions (which in particular
can be the Gauss and the mean curvature) satisfying 
a partial differential equation equivalent to the Gauss equation.  
\end{abstract}
\maketitle




\section{Introduction}
\vspace{0.5cm}

A fundamental problem in  differential geometry is to characterize an object by
functions naturally connected with the object. A classical and well known example 
in this kind is the characterization of a curve in the space by its curvature 
and torsion as functions of the natural parameter.
Another example is the Bonnet's classical theorem, according to which a surface in 
$\mathbb R^3$ is determined by the coefficients of the first and the second fundamental
forms satisfying the equations of Gauss and Codazzi.

In these two examples we can see a principal difference in the investigations of curves and
of surfaces -- in the second case we need a relation between the function determining
the surface. This seems very natural according to the Gauss and Codazzi equations.
Another difference in these two examples is that in the first one the curvature and the torsion 
of a curve are invariants of the curve, while the coefficients of the fundamental forms of a surface are not.
The clarification of the role of invariants in the determination of surfaces begins in \cite{Bonnet},
where Bonnet studies the relationship between the surface and its first fundamental form and 
principal curvatures. For other works in this direction see e.g. \'E. Cartan \cite{Cartan}, S.-S. Chern  
\cite{Chern}, F. Rellich \cite{Rellich}, W. Scherrer \cite{Scherrer}.
  
Recently many works have investigated the problem of determining the surfaces of special types in the 3- or 4-dimensional Euclidean space
as well as in pseudo-Euclidean spaces by invariants. G. Ganchev and V. Mihova \cite{Ganchev-Mihova} proved that a Weingarten surface in 
$\mathbb R^3$ is determined up to a motion in special parameters 
by three functions (one of which is invariant and the two other determine the Weingarten nature of the surface)
which satisfy
a natural PDE. By the same arguments a similar result  holds for space-like Weingarten surfaces as well
as  for time-like Weingarten surfaces in the Lorentz space $\mathbb R^3_1$, if the Gauss curvature $K$ and the mean curvature $H$
satisfy $K-H^2<0$, see \cite{G-V-2013}. 
R. Tribucy and I. Guadalupe \cite{Tribucy-Guadalupe} proved that 
a minimal surface in $\mathbb R^4$ is determined up to a motion by the Gauss and the normal curvature satisfying a system of two
PDEs. Similar problems for minimal surfaces in $\mathbb R^4_1$ and $\mathbb R^4_2$ are considered respectively in 
\cite{Ganchev-Vel} and \cite{Ganchev-Krassi}. 

In the case of an arbitrary non-umbilical surface in $\mathbb R^3$, the present author proved  that
the surface admits special principal parameters and the surface is uniquely determined (up to a motion)
in these parameters by the Gauss and the mean curvature functions satisfying a PDE, see \cite{OK}.  The same question
is considered in  \cite{K-K-M} for surfaces with non-vanishing $K-H^2$ in $\mathbb R^3_1$. The arguments and the results
use isotropic (or null) parameters and the main result contains an integro-differential equation, that is not usual 
and not very convenient.  On the other hand, the considerations of space-like  surfaces and
of time-like surfaces in the $\mathbb R^3_1$ with $K-H^2<0$ are absolutely similar to the case of surfaces in $\mathbb R^3$.
So only the case of time-like surfaces satisfying $K-H^2>0$ (so, admitting real
asymptotic lines) is considerably different and deserves a special interest.

Exactly this case is the object of the present paper. It turns out that the asymptotic parameters are very
useful in our case. We prove  that the surface is uniquely determined in asymptotic parameters
by four invariant geometrically defined functions, satisfying a system of three PDSs. Then
we show that the surface admits special asymptotic parameters (we call them canonical asymptotic parameters).
In the particular case, when the surface is minimal our definition of canonical asymptotic parameters coincides with 
the definition of canonical parameters given in \cite{Ganchev-Krassi-2019}.
We obtain a fundamental theorem, stating that a time-like surface in $\mathbb R_1^3$ 
with $K-H^2>0$ is determined in canonical asymptotic parameters up to a motion in the space by 
its Gauss and mean curvature functions, satisfying a PDE , equivalent to the Gauss equation. 
When the Gauss curvature is constant this equation is studied in \cite{G-H-I}.

\vspace{0.5cm}
 



\section{Preliminaries}

Let $\mathbb R^3_1$ be the 3-dimensional pseudo-Euclidean space with scalar product given by
$$
	\langle x,y \rangle=x_1y_1+x_2y_2-x_3y_3  \ .
$$
A vector $x$ is called space-like, time-like or light-like if $\langle x,x \rangle$ is $>0$, 
$<0$ or $=0$, respectively. A surface  $S$ in $\mathbb R^3_1$ is said to be space-like or time-like 
if at any point of $S$ its normal vector is time-like, resp. space-like. Then the induced metric on $S$
is respectively Riemannian or Lorentzian. 

Suppose the time-like surface $S$ is defined by $	z=z(u,v) $ and denote the derivatives of $z(u,v)$ by
$	z_u $, $z_v$, $z_{uu}$, etc. The coefficients of the first fundamental form are given by 
$$
	E=z_u^2  \hspace{1.5cm}  F=\langle z_u,z_v \rangle   \hspace{1.5cm}  G=z_v^2  \ .
$$
We denote by $n$ the unit vector, collinear to the Lorentzian cross product $z_u\times z_v$, so that
the triple $\{ z_u,z_v,n \}$ is positively oriented. Then the coefficients of  the second fundamental form 
are given by
$$
	L=\langle n,z_{uu} \rangle  \hspace{1.5cm}   M=\langle n,z_{uv} \rangle \hspace{1.5cm}  N=\langle n,z_{vv} \rangle \ .
$$
The Gauss curvature $K$ and the mean curvature $H$ of $S$ are defined by
$$
	K=\frac{LN-M^2}{EG-F^2}    \hspace{2cm}  H=\frac{EN-2FM+GL}{2(EG-F^2)} \ .
$$

When $E=G=0$, the parameters $(u,v)$ are called isotropic or null parameters. For them the parametric 
lines are isotropic. In \cite{K-K-M} the geometry of a time-like
surface is considered using isotropic parameters. On the other hand, in the study of space-like surfaces or  time-like surfaces  with
$K-H^2<0$, the principal parameters are more natural and more useful. For these parameters the parametric lines are principal and
they are characterized by $F=M=0$. Then the theory of the surface can be developed 
analogously to that in the classical Riemannian geometry, see \cite{G-V-2013}.  In our study we shall use asymptotic parameters, i.e. these for which 
the parametric lines are asymptotic.  Analytically they are characterized by $L=N=0$. 
Two pairs of asymptotic parameters $(u, v)$ and $(\bar u,\bar v)$  in
a neighbourhood of a  point are related either by $u = u(\bar u)$, $v = v(\bar v)$, or
$u = u(\bar v)$, $v = v(\bar u)$.

It is easy to see that if the Gauss curvature $K$ is positive, the time-like surface admits real asymptotic parameters.
Suppose moreover that $K-H^2$ never vanishes. As we said, in the case $K-H^2<0$
the surface can be parametrized by (real) principal parameters and the theory is analogous 
to that in the Riemannian case. So suppose $K-H^2>0$, i.e. the principal curvatures are imaginary. Then
in asymptotic parameters 
$$
	K-H^2=-EG\frac{M^2}{(EG-F^2)^2} \ .
$$
Without loss of generality we assume  $E>0$, $G<0$. We put
$$
	x=\frac{z_u}{\sqrt E}  \hspace{1.5cm}   y=\frac{z_v}{\sqrt{-G}}   \hspace{1.5cm}  a=\langle x,y \rangle \ .
$$
Then $x^2=1$, $y^2=-1$. Denoting by $\nabla$ the covariant differentiation on $\mathbb R^3_1$ 
we obtain the following Frenet-type formulas:

\begin{equation} \label{eq:FrenetFormulas}
	\begin{array}{l}
		\ds\nabla_xx=-a\gamma_1x+\gamma_1y \\
		\ds\nabla_xy=\left(\frac{x(a)}{1+a^2}+\gamma_1\right)x+a\left(\frac{x(a)}{1+a^2}+\gamma_1\right)y  +\alpha n \\
		\ds\nabla_xn=-\frac{a\alpha}{1+a^2}x+\frac{\alpha}{1+a^2}y     \\
		\ds\nabla_yx=a\left(\frac{y(a)}{1+a^2}-\gamma_2\right)x-\left(\frac{y(a)}{1+a^2}-\gamma_2\right)y+\alpha n   \\
		\ds\nabla_yy=\gamma_2x+a\gamma_2y  \\		
		\ds\nabla_yn=-\frac{\alpha}{1+a^2}x-\frac{a\alpha}{1+a^2}y
	\end{array}
\end{equation}
where 
\begin{equation} \label{eq:Gamma'sFormulas}
	\begin{array}{l}
		\ds\gamma_1=-\frac1{1+a^2} \left( x(a)+a\frac{(\sqrt{-G})_u}{{\sqrt E\sqrt{-G}}}-\frac{(\sqrt E)_v}{\sqrt E\sqrt{-G}} \right) \\
		\ds\gamma_2=\frac1{1+a^2} \left( y(a)+a\frac{(\sqrt E)_v}{\sqrt E\sqrt{-G}}+\frac{(\sqrt{-G})_u}{\sqrt E\sqrt{-G}} \right) 
	\end{array}
\end{equation}
$$
	a=\frac{F} {\sqrt E\sqrt {-G}}           \hspace{1.5cm}    \alpha=\frac M {\sqrt E\sqrt {-G}}\ .
$$

Hence we obtain also 
\begin{equation} \label{eq:KandHbyaAndAlpha}
	K=\frac{\alpha^2}{1+a^2}    \hspace{1.5cm}    H=\frac{a\alpha}{1+a^2} \ .
\end{equation}

\begin{remark}
Note that the functions $a$, $\gamma_i$ and $\alpha$ are invariants of the surface. 
 But regardless of the notations
 $\gamma_i$ and $\alpha$ are not in general the geodesic curvatures  and torsion
of the asymptotic lines. More precisely
if $\overline\gamma_i$ are the geodesic curvatures of the asymptotic lines, then 
$$
	\overline\gamma_i=\gamma_i\sqrt{1+a^2} \ .
$$
Similarly, if $\overline\alpha$ is  the torsion of the asymptotic lines, then
$$
	\alpha=\overline\alpha\sqrt{1+a^2} \ .
$$
In particular, if $a=0$, i.e. if the surface is minimal, then $\overline\gamma_i=\gamma_i$
and $\overline\alpha=\alpha$.
\end{remark}

We shall call $\gamma_1$, $\gamma_2$, $a$ and $\alpha$  {\it the basic asymptothic invariants} of the surface.

Using (\ref{eq:FrenetFormulas}) we derive the main equations of the surface - the Gauss equation
\begin{equation} \label{eq:GaussEquation-1}
	\begin{array}{l}
		\ds x(\gamma_2)-y(\gamma_1)-2a\gamma_1\gamma_2 -\gamma_1^2+\gamma_2^2
			-\gamma_1 \frac{x(a)}{1+a^2}	-\gamma_2\frac{y(a)}{1+a^2} \\
			\ds -\frac{a_{uv}}{(1+a^2)\sqrt E\sqrt {-G}}+\frac{ax(a)y(a)}{(1+a^2)^2}
		+\frac{\alpha^2}{1+a^2}=0    \vspace{0.3cm}\\
	\end{array}
\end{equation}
and the equations of Codazzi:
\begin{equation} \label{eq:EquationsOfCodazzi}
	\begin{array}{l}
		\ds x(\alpha)= a\alpha \frac{x(a)}{1+a^2}  +2 \alpha\left(\frac{y(a)}{1+a^2}-\gamma_2 \right)  \ , \\
		\ds y(\alpha)= a\alpha \frac{y(a)}{1+a^2}-2\alpha \left(\frac{x(a)}{1+a^2}+\gamma_1\right)   \ .
	\end{array}
\end{equation}
Note that the expression $\ds \frac{\alpha^2}{1+a^2}$ in the Gauss equation is exactly the Gauss curvature $K$.

\vspace{0.5cm}



\setcounter{equation}{0}

\section{Determining a time-like surface by invariants}

Using (\ref{eq:Gamma'sFormulas}), from the Codazzi equations (\ref{eq:EquationsOfCodazzi}) we find

\begin{equation}\label{eq:System}
	\begin{array}{rl}
	 \left(\sqrt E\right)_v \hspace{-0.2cm}    &  \ds =f_v\sqrt E +af_u\sqrt{-G}  \\
	 \left(\sqrt{-G}\right)_u  \hspace{-0.3cm} &  \ds =-af_v\sqrt E +f_u\sqrt{-G}
	\end{array}
\end{equation} 
where
$$
	f=\frac12 \left( \log\sqrt{1+a^2}-\log|\alpha| \right)  \,.
$$

\begin{remark} \label{remark:K-aAndalppha}
Since $K=\frac{\alpha^2}{1+a^2}$ we have also $f=-\log\root4\of  K$.
\end{remark}

Now the formulas for $\gamma_1$ and $\gamma_2$ take the form
\begin{equation} \label{eq:gamma1-gamma2}
	\begin{array}{l} 
		\ds \gamma_1	=-\frac{a_u}{(1+a^2)\sqrt E} +\frac{f_v}{\sqrt{-G}} \,, \\
		\ds \gamma_2=\frac{a_v}{(1+a^2)\sqrt{-G}}   + \frac{f_u}{\sqrt E}  \,.  
	\end{array}  
\end{equation}
From this system for $\sqrt E$ and $\sqrt{-G}$ we obtain

\begin{equation} \label{eq:firstfundformByInvariants}
	\begin{array}{l}
		\ds \sqrt E=\frac{a_ua_v+(1+a^2)^2f_uf_v}{(1+a^2)(-a_v\gamma_1+(1+a^2)f_v\gamma_2)} \ ,  \vspace{0,2cm}\\
		\ds \sqrt{-G}=\frac{a_ua_v+(1+a^2)^2f_uf_v}{(1+a^2)(a_u\gamma_2+(1+a^2)f_u\gamma_1)}\ .
	\end{array}  
\end{equation}
In particular (\ref{eq:firstfundformByInvariants}) implies
that the functions on the right are positive.

The Frenet-type formulas (\ref{eq:FrenetFormulas}) can be written as
$$
	\begin{array}{l}
		\ds x_u=\sqrt E\gamma_1(-ax+y)  \\
		\ds y_u=\sqrt E\left[\left(\frac{x(a)}{1+a^2}+\gamma_1\right)(x+ay)  +\alpha n \right]\\
		\ds n_u=\sqrt E\frac{\alpha}{1+a^2}(-ax+y)    \\
		\ds x_v=\sqrt{-G}\left[\left(\frac{y(a)}{1+a^2}-\gamma_2\right)(ax-y)+\alpha n \right]  \\
		\ds y_v=\sqrt{-G}\gamma_2(x+ay  ) \\		
		\ds n_v=-\sqrt{-G}\frac{\alpha}{1+a^2}(x+ay) \ .
	\end{array}
$$

Using this we shall prove the following theorem

\begin{thm} \label{eq:Bonnet-typeTheorem}

Assume the functions $\gamma_1(u,v)$, $\gamma_2(u,v)$, $a(u,v)$, $\alpha(u,v)$ are smooth in a neighbourhood 
$ D $ of a point $(u_0,v_0) \in \mathbb R^2$. Suppose the functions
$ f=\frac12 \left( \log\sqrt{1+a^2}-\log|\alpha| \right) $,
$$
	\Phi= \frac{a_ua_v+(1+a^2)^2f_uf_v}{(1+a^2)(-a_v\gamma_1+(1+a^2)f_v\gamma_2)}   \hspace{1.5cm}
		 \Psi=\frac{a_ua_v+(1+a^2)^2f_uf_v}{(1+a^2)(a_u\gamma_2+(1+a^2)f_u\gamma_1)}
$$
are well defined and $\Phi>0$, $\Psi>0$ in $D$. If the following conditions are satisfied:
\begin{equation} \label{eq:GaussEquation-2}
	\begin{array}{l}
		\ds \frac{(\gamma_2)_u}{\Phi}-\frac{(\gamma_1)_v}{\Psi}-2a\gamma_1\gamma_2 -\gamma_1^2+\gamma_2^2
			-\gamma_1 \frac{a_u}{(1+a^2)\Phi}	-\gamma_2\frac{a_v}{(1+a^2)\Psi} \\
			\ds -\frac{a_{uv}}{(1+a^2)\Phi\Psi}+\frac{aa_ua_v}{(1+a^2)^2\Phi\Psi}
		+\frac{\alpha^2}{1+a^2}=0    \vspace{0.3cm}    \ , \\
	\end{array}
\end{equation}
\begin{equation} \label{eq:EquationsOfCodazzi-2}
	\begin{array}{l}
		\ds (\log\Phi)_v+af_u\frac{a_v\gamma_1-(1+a^2)f_v\gamma_2}{a_u\gamma_2+(1+a^2)f_u\gamma_1}=f_v \ , \\
		\ds (\log\Psi)_u+af_v\frac{a_u\gamma_2+(1+a^2)f_u\gamma_1}{-a_v\gamma_1+(1+a^2)f_v\gamma_2}=f_u \ ,
	\end{array}
\end{equation}  
then there exists a unique (up to a motion) surface in $\mathbb R^3_1$ with basic
asymptotic invariants the given functions $\gamma_1$, $\gamma_2$, $a$, $\alpha$.
Moreover $(u,v)$ are asymptotic parameters.
\end{thm}

\begin{proof}

We shall search for unit vector fields $x(u,v)$, $y(u,v)$, $n(u,v)$, subjects to the system
\begin{equation} \label{eq:MainEquationsInTheBonetTheorem}
	\begin{array}{l}
		\begin{array}{l}
		\ds x_u=\gamma_1\Phi(-ax+y)  \\
		\ds y_u=\left(\frac{a_u}{1+a^2}+\gamma_1\Phi\right)(x+ay)  +\alpha\Phi n \\
		\ds n_u=\frac{\alpha\Phi}{1+a^2}(-ax+y)    \\
		\ds x_v=\left(\frac{a_v}{1+a^2}-\gamma_2\Psi\right)(ax-y)+\alpha\Psi n   \\
		\ds y_v=\gamma_2\Psi(x+ay ) \\		
		\ds n_v=-\frac{\alpha\Psi}{1+a^2}(x+ay) 
	\end{array}
	\end{array}
\end{equation}
or briefly
$$
	\xi_u=U\xi \hspace{2cm} \xi_v=V\xi  \,,
$$
where
$\xi=(x,y,n)^t$,
$$
	U=\Phi\left(
			\begin{array}{ccc}
				\ds -a\gamma_1    & \gamma_1   &   0  \\
				\ds \frac{a_u}{(1+a^2)\Phi}+\gamma_1   & \ds \frac{aa_u}{(1+a^2)\Phi}+a\gamma_1  & \alpha  \\
				\ds -\frac{a\alpha}{1+a^2}   & \ds \frac{\alpha}{1+a^2}   &   0  \\
				\end{array}
		\right)      \,,
$$
$$
	V=\Psi\left(
			\begin{array}{ccc}
				\ds \frac{aa_v}{(1+a^2)\Psi}-a\gamma_2  & \ds \frac{-a_v}{(1+a^2)\Psi}+\gamma_2  & \alpha   \\
				\ds \gamma_2 & a\gamma_2 &  0  \\		
				\ds -\frac{\alpha}{1+a^2}  &  \ds -\frac{a\alpha}{1+a^2} &  0
			\end{array}
		\right) \,.
$$
The integrability condition of this system is
$$
	U_v-V_u=VU-UV \ .
$$
Using the condition of the theorem we can conclude that this condition is fulfilled.
Consequently there exist unique vector fields $x$, $y$, $n$, defined in a
subset $D_0$ of $D$ such that
$$
	x(u_0,v_0)=x_0 \qquad y(u_0,v_0)=y_0  \qquad n(u_0,v_0)=n_0 \ ,
$$ 
where $(u_0,v_0)$ is a point in $D$ and $\{x_0,y_0,n_0\}$ is a positively oriented orthonormal basis in a point $z_0$ of $\mathbb R^3_1$, such that 
$$
	x_0^2=1   \qquad   x_0y_0=a(u_0,v_0) \qquad y_0^2=-1   \qquad   n_0^2=1   \qquad  x_0n_0=0  \qquad  y_0n_0=0   \ .
$$
A standard procedure leads to the conclusion that 
$$
	x^2=1   \qquad   xy=a \qquad y^2=-1  \qquad  xn=0  \qquad  yn=0    \qquad   n^2=1  
$$
because this is true at $(u_0,v_0)$. 
The triple $\{x,y,n\}$ is positively oriented, because this is so at $(u_0,v_0)$.

Now we look for a vector valued function $z(u,v)$, such that
$$
	z_u=\Phi x  \qquad\qquad z_v=\Psi y \ .
$$
The integrability condition 
$$
	(\Phi x)_v=(\Psi y)_u 
$$
is fulfilled and hence there exists a unique function $z(u,v)$, 
defined in a sub-domain of $D_0$, such that $z(u_0,v_0)=z_0$.

On the other hand, as  $x$, $y$, $n$ satisfy (\ref{eq:MainEquationsInTheBonetTheorem}), 
then  $(u,v)$ are asymptotic parameters and the given functions 
$\gamma_1$, $\gamma_2$, $a$, $\alpha$ are the basic asymptotic invariants
of the surface. 

\end{proof}

\vspace{0.5cm}



\setcounter{equation}{0}

\section{Canonical parameters for surfaces with positive Gauss curvature}

Formulas (\ref{eq:System}) imply easily that for any point $(u_0,v_0)$  the functions
\begin{equation} \label{eq:phi-psi}
	\begin{array}{l} 
		\vspace{2mm}
		\ \ds\sqrt Ee^{-\ds\hspace{-1mm}\int_{v_0}^v\left(f_v+af_u\frac{\sqrt{-G}}{\sqrt E}\right)dv
		                -\hspace{-1mm}\int_{u_0}^u\left(-af_v\frac{\sqrt E}{\sqrt{-G}}+f_u\right)\hspace{-1mm}(u,v_0)du-\hspace{-1mm}f(u_0,v_0)} \ ,  \\
		\vspace{2mm}
		\ds\sqrt {-G}e^{-\ds\hspace{-1mm}\hspace{-1mm}\int_{u_0}^u\left(-af_v\frac{\sqrt E}{\sqrt{-G}}+f_u\right)du
		                -\int_{v_0}^v\left(f_v+af_u\frac{\sqrt{-G}}{\sqrt E}\right)\hspace{-1mm}(u_0,v)dv-\hspace{-1mm}f(u_0,v_0)}  \ ,          
	\end{array}   
\end{equation}
do not depend on $v$ and $u$, respectively. 
To work with the most natural simplification of these functions, we give the following definition: 

\begin{defn} \label{D:def-can}
 We say that the asymptotic parameters $(u,v)$ of a surface of positive Gauss curvature in $\R^3_1$ are {\it canonical asymptotic parameters}
if the functions $\varphi(u)$ and $\psi(v)$ defined by \eqref{eq:phi-psi} are equal to 1.  
\end{defn}

\begin{def} \label{D:def-can}
 We say that the asymptotic parameters $(u,v)$ of a surface of positive Gauss curvature in $\R^3_1$ are {\it canonical asymptotic parameters}
if the functions $\varphi(u)$ and $\psi(v)$ defined by \eqref{eq:phi-psi} are equal to 1.  
\end{def}

In other words the  asymptotic parameters are canonical if
$$
	\begin{array}{l} 
		\vspace{2mm}
		\ds \sqrt E=e^{\ds \int_{v_0}^v\left(f_v+af_u\frac{\sqrt{-G}}{\sqrt E}\right)dv
		                +\hspace{-1mm}\int_{u_0}^u\left(-af_v\frac{\sqrt E}{\sqrt{-G}}+f_u\right)\hspace{-1mm}(u,v_0)du +\hspace{-1mm}f(u_0,v_0)} \ ,  \\
		\vspace{2mm}
		\ds \sqrt{-G}=e^{\ds \int_{u_0}^u\left(-af_v\frac{\sqrt E}{\sqrt{-G}}+f_u\right)du
		                +\hspace{-1mm}\int_{v_0}^v\left(f_v+af_u\frac{\sqrt{-G}}{\sqrt E}\right)\hspace{-1mm}(u_0,v)dv  +\hspace{-1mm}f(u_0,v_0)} \ .
	\end{array}   
$$

\begin{remark}In general the definition depends on the point $(u_0,v_0)$, so it is more exact
to speek about {\it canonical asymptotic parameres about a fixed point}.
\end{remark}

\begin{prop} \label{eq:ExistenceCanonicalParameters}
Let $S$ be a surface with positive Gauss curvature in $\mathbb R^3_1$. Then locally $S$ admits 
canonical asymptotic parameters.
\end{prop}

\begin{proof}

Taking arbitrary  asymptotic parameters $(u,v)$ we define new parameters $\bar u$, $\bar v$ by 
$$
	\bar u=\ds \int_{u_0}^u\varphi(u) +u_0 \ , \hspace{1.5cm}
		\bar v=\ds \int_{v_0}^v\psi(v) +v_0 \ .
$$ 
Then  $\bar u=\bar u(u)$ and $\bar v=\bar v(v)$, so the new parameters 
$\bar u$, $\bar v$ are also asymptotic.  On the other hand
$$
	\bar u_u=\sqrt Ee^{-\ds\int_{v_0}^v\left(f_v+af_u\frac{\sqrt {-G}}{\sqrt E}\right)dv
		                -\int_{u_0}^u\left(-af_v\frac{\sqrt {E}}{\sqrt{-G}}+f_u\right)(u,v_0)du-f(u_0,v_0)} \ ,
$$
$$
	\bar v_v=\sqrt {-G}e^{-\ds\int_{u_0}^u\left(-af_v\frac{\sqrt E}{\sqrt G}+f_u\right)du
		                -\int_{v_0}^v\left(f_v+af_u\frac{\sqrt {-G}}{\sqrt E}\right)(u_0,v)dv-f(u_0,v_0)} \ 
$$
imply
$$
	\overline E =e^{\ds2\left( \int_{v_0}^v\left(f_v+af_u\frac{\sqrt{-G}}{\sqrt E}\right)dv
		                +\int_{u_0}^u\left(-af_v\frac{\sqrt {E}}{\sqrt {-G}}+f_u\right)(u,v_0)du+f(u_0,v_0) \right)}  \ .
$$  
Moreover a straightforward verification gives
$$
	\ds \left(f_v+af_u\frac{\sqrt {-G}}{\sqrt {E}}\right)(u,v)
				=\left(\bar f_{\bar v}+\bar a\bar f_{\bar u}\frac{\sqrt{-\overline G}}{\sqrt {\overline E}}\right)(\bar u(u),\bar v(v))\frac{d\bar v}{dv} \ ,
$$
$$
	\ds \left(-af_v\frac{\sqrt {E}}{\sqrt {-G}}+f_u\right)(u,v)
				=\left(-\bar a\bar f_{\bar v}\frac{\sqrt {\overline E}}{\sqrt { {-\overline G}}}+\bar f_{\bar u}\right)(\bar u(u),\bar v(v))\frac{d\bar u}{du} \ .
$$
Denote $\bar u_0=\bar u(u_0)=u_0$, $\bar v_0=\bar v(v_0)=v_0$. Hence we derive
$$
	\overline E=e^{\ds 2\left( \int_{\bar v_0}^{\bar v} \hspace{-1mm}\left(\bar f_{\bar v}
																+\bar a\bar f_{\bar u}\frac{\sqrt {{-\overline G}}}{\sqrt {\overline E}}\right)(\bar u,\bar v)d\bar v
		                +\hspace{-1mm}\int_{\bar u_0}^{\bar u}\hspace{-1mm}\left(\hspace{-1mm}-\bar a\bar f_{\bar v}\frac{\sqrt {\overline E}}{\sqrt {{-\overline G}}}
										+ \bar f_{\bar u}\right)(\bar u,\bar v_0)d\bar u  +f(\bar u_0,\bar v_0) \right)} 
$$
i.e. $\bar\varphi (\bar u)=1$.
Analogously,  $\bar\psi(\bar v)=1$. So, the parameters $(\bar u,\bar v)$ are canonical.

\end{proof}

The following assertion gives the connection between different pairs of canonical
asymptotic parameters.

\begin{lemma} If $(u,v)$ and $(\bar u,\bar v)$ are canonical asymptotic parameters
in a neighbourhood of a point $p$, then
$$
	\left\{\begin{array}{l}
	 \bar u=\pm u +u_1 \\
   \bar v=\pm v+v_1
	\end{array}\right.
\qquad
or
\qquad
	\left\{\begin{array}{l}
		\bar u=\pm v+v_1 \\
		\bar v=\pm u+u_1 
	\end{array}\right.	
$$
for some constants $u_1$, $v_1$.
\end{lemma}

Assume now that $S$ is
parametrized by canonical asymptotic parameters. Then  the Gauss equation
(\ref{eq:GaussEquation-1}), can be written in the form
$$
	 \frac1{\sqrt E\sqrt{-G}}\left(\frac{a_{uv}}{1+a^2}-\frac{aa_ua_v}{(1+a^2)^2} -2af_uf_v \right)
	   + \frac{  aa_uf_u}{(1+a^2)E}	   +\frac{  aa_vf_v}{(1+a^2)G}
$$
$$
	+ \frac{(f_u)^2}{ E}+\frac{(f_v)^2}{G} 
	+ \frac1{\sqrt E}\left(\frac{f_u}{\sqrt E}\right)_u -\frac1{\sqrt{-G}}\left(\frac{f_v}{\sqrt {-G}}\right)_v 
	 +\frac{\alpha^2}{1+a^2}=0  \ .
$$

Hence we can prove the following  theorem:

\begin{thm} \label{eq:MainTheoremByAlphaAndA}

Assume the functions $a(u,v)$, $\alpha(u,v)$ are smooth in a neighbourhood 
$ D $ of a point $(u_0,v_0) \in \mathbb R^2$, $\alpha>0$ in $ D $. 
Define $f=\frac12 \left( \log\sqrt{1+a^2}-\log\alpha \right)$.
Let $(\Phi,\Psi)$ be the solution of the  PDE system
\begin{equation} \label{eq:PDEsystemInMainTheorem}
	\left\{\begin{array}{rl}
	 \Phi_v \hspace{-0.2cm}    &  \ds =f_v\Phi +af_u\Psi  \\
	 \Psi_u  \hspace{-0.3cm} &  \ds =-af_v\Phi +f_u\Psi
	\end{array}\right.
\end{equation}
with initial conditions 
$$
	\Phi(u,v_0)=e^{\ds\int_{u_0}^u\left(-af_v+f_u\right)(u,v_0)du +f(u_0,v_0)} \,,
$$
$$
	\quad  \Psi(u_0,v)=e^{\ds\int_{v_0}^v\left(f_v+af_u\right)(u_0,v)dv +f(u_0,v_0)} \,.
$$
If the equation
\begin{equation} \label{eq:GaussEquationMainTheorem}
	\begin{array}{l}
		\ds \frac1{\Phi\Psi}\left(\frac{a_{uv}}{1+a^2}-\frac{aa_ua_v}{(1+a^2)^2} -2af_uf_v \right)
				+ \frac{  aa_uf_u}{(1+a^2)\Phi^2}	   +\frac{  aa_vf_v}{(1+a^2)\Psi^2 } \\
		\ds	\qquad + \frac{(f_u)^2}{\Phi^2}+\frac{(f_v)^2}{\Psi^2} 
				+ \frac1{\Phi}\left(\frac{f_u}{\Phi}\right)_u -\frac1{\Psi}\left(\frac{f_v}{\Psi}\right)_v 
				+\frac{\alpha^2}{1+a^2}=0  
	\end{array}
\end{equation}
is satisfied, then there exists a unique (up to a motion) surface in $\mathbb R^3_1$ with basic 
asymptotic invariants the given functions $a$, $\alpha$.
Moreover $(u,v)$ are canonical asymptotic parameters.
\end{thm}

\begin{proof}
We introduce the functions
\begin{equation} \label{eq:defGammas}
	\begin{array}{l}
		\ds\gamma_1=-\frac{a_u}{(1+a^2)\Phi}+\frac{f_v}{\Psi}  \ ,  \qquad\qquad
		\ds\gamma_2=\frac{a_v}{(1+a^2)\Psi}+\frac{f_u}{\Phi}  \ . 
	\end{array}
\end{equation}
Using (\ref{eq:PDEsystemInMainTheorem}), (\ref{eq:GaussEquationMainTheorem}) and (\ref{eq:defGammas}) 
we can check that the conditions of Theorem \ref{eq:Bonnet-typeTheorem} are fulfilled. 
Hence there exists
a unique (up to a motion) surface in $\mathbb R^3_1$ with basic 
asymptotic invariants  $\gamma_1$, $\gamma_2$, $a$, $\alpha$ and $(u,v)$ are asymptotic parameters.

Finally, since the pair $(\Phi,\Psi)$ is a solution
of the Cauchy problem in the condition of the theorem, 
then $(u,v)$ are canonical asymptotic parameters. 

\end{proof}

To obtain the above theorem in a form with the classical invariants $K$ and $H$, we first remember that
the Gauss curvature $K$ and the mean curvature $H$ of a surface, parametrized with asymptotic
parameters satisfy:
$$
	K=\frac{\alpha^2}{1+a^2} \ ,  \hspace{1.5cm}
	H=\frac{a\alpha}{1+a^2} \ . 
$$
Conversely, 
$$
	a=\frac{H}{\sqrt{K-H^2}}  \qquad\qquad\  \alpha =\frac{K}{\sqrt{K-H^2}}
$$
or
$$
	\ \ a=-\frac{H}{\sqrt{K-H^2}}  \qquad\qquad \alpha =-\frac{K}{\sqrt{K-H^2}} \ .
$$
Consequently,  from Theorem \ref{eq:MainTheoremByAlphaAndA} we obtain:

\begin{thm} \label{eq:MainTheoremByKandH}
Assume the functions $K(u,v)$, $H(u,v)$ are smooth in a neighbourhood 
$ D $ of a point $(u_0,v_0) \in \mathbb R^2$, $K>0$, $K-H^2>0$ in $U$. 
Define  $a=\frac{H}{\sqrt{K-H^2}}$.
Let $(\Phi,\Psi)$ be the solution of the  PDE system
$$
	\left\{\begin{array}{rl}
	 \Phi_v \hspace{-0.2cm}    &  \ds =-\frac1{4K}\left( K_v\Phi +aK_u\Psi \right)  \\
	 \Psi_u  \hspace{-0.3cm} &  \ds =\frac{1}{4K}\left( aK_v\Phi -K_u\Psi \right)
	\end{array}\right.
$$
with initial conditions 
$$
	\Phi(u,v_0)=e^{\ds\int_{u_0}^u\left(a(\log\root4\of K)_v-(\log\root4\of K)_u\right)(u,v_0)du -\log\root4\of K(u_0,v_0)} \,,
$$
$$
	\quad  \Psi(u_0,v)=e^{\ds\int_{v_0}^v\left(-(\log\root4\of K)_v-a(\log\root4\of K)_u\right)(u_0,v)dv -\log\root4\of K(u_0,v_0)} \,.
$$
If the equation
$$
	 \frac1{\Phi\Psi}\left(\frac{a_{uv}}{1+a^2}-\frac{aa_ua_v}{(1+a^2)^2} -\frac{aK_uK_v}{8K^2} \right)
	   - \frac{  aa_uK_u}{4(1+a^2)K\Phi^2}	   -\frac{  aa_vK_v}{4(1+a^2)K\Psi^2}
$$
$$
	+ \frac{(K_u)^2}{16\Phi^2K^2}+\frac{(K_v)^2}{16\Psi^2K^2} 
	- \frac1{4\Phi}\left(\frac{K_u}{\Phi K}\right)_u +\frac1{4\Psi}\left(\frac{K_v}{\Psi K}\right)_v 
	 +K=0  
$$
is satisfied, then there exists a unique (up to a motion) surface in $\mathbb R^3_1$ with 
Gauss curvature $K$ and mean curvature $H$.
Moreover $(u,v)$ are canonical asymptotic parameters.

\end{thm}

In particular, if $S$ is minimal, our notations imply $a=0$, $K=\alpha^2$,
$f=-\log\sqrt\alpha$.
Definition \ref{D:def-can} of canonical asymptotic parameters is now equivalent to

$$
		 E=\frac{1}{\alpha} \ ,  \qquad\qquad  F=0 \ ,  \qquad\qquad
		 G=-\frac{1}{\alpha}  \ ,
$$
$$
		 L=0 \ ,  \qquad\qquad  M^2=1 \ ,  \qquad\qquad
		 N=0  \ .
$$
Consequently in this case our definition 
of canonical asymptotic parameters coincides with the definition of canonical parameters, given in \cite{Ganchev-Krassi-2019}.

Let now $S$ be a surface of constant positive Gauss curvature and imaginary principal curvatures. 
Up to similarity we suppose that the Gauss curvature $K=1$. 
Then Theorem \ref{eq:MainTheoremByKandH} implies
$$
	\Phi_v=0 \hspace{2cm}  \Psi_u=0 \hspace{2cm}   \Phi(u,v_0)=1   \hspace{2cm}   \Psi(u_0,v)=1 \ ,
$$
so $\Phi=\Psi=1$ and the Gauss equation takes the form
\begin{equation} \label{eq:SurfaceConstantCurvatureEquation}
	 \frac{a_{uv}}{1+a^2}-\frac{aa_ua_v}{(1+a^2)^2}  +1=0    
\end{equation}
where
$$
	a=\frac{H}{\sqrt{1-H^2}} \ .
$$	
If we put $a=\sinh \omega$ in (\ref{eq:SurfaceConstantCurvatureEquation}) we obtain the equation of the surface of Gauss curvature 1
in the form of  a cosh-Gordon equation as in \cite{G-H-I}:
$$
	\omega_{uv}+\cosh\omega = 0 \ .
$$
In  accordance with the general notion of Chebyshef coordinates (see e.g. \cite{Gray-A-S}),  
$(u,v)$ are called in  \cite{G-H-I} asymptotic Chebyshev coordinates.

\vspace{0.5cm}


\setcounter{equation}{0}

\section{Examples}

\begin{example} \label{ex:2}  Consider the time-like surface in $\mathbb R^3_1$  with the parametrization
$$
	z(u,v)=\Big( \frac{v^3}6+\frac{u^2v}2-\frac v2,\frac{u^2}2+\frac{v^2}2,\frac{u^3}6+\frac{uv^2}2+\frac{v^3}6+\frac u2 \Big) \ .
$$
The coefficients of the first and the second fundamental forms are respectively
$$
	E= - \frac14(1-u^2+v^2)^2 \qquad F=0  \qquad G= \frac14(1-u^2+v^2)^2  \ ,
$$
$$
	L=-1  \qquad M=0  \qquad  N=1 \ .
$$
For the Gauss curvature and the mean curvature we find 
$$
	K=-\frac{16}{(1-u^2+v^2)^4}<0 \qquad  H=0   \ .
$$
Hence the surface has (real) principal parameters and imaginary asymptotic parameters. 
As explained in Introduction, for such surfaces the canonical (principal) parameters are defined 
as in \cite{G-V-2013} following the procedure for surfaces in $\mathbb R^3$.

This surface is an Enneper-type time-like surface of negative Gauss curvature.

\end{example}

\begin{example} \label{ex:1}  For the time-like surface in $\mathbb R^3_1$  with the parametrization
$$
	z(u,v)=\Big( \frac{u^3}6+\frac{uv^2}2-\frac u2,uv,\frac{u^2v}2+\frac{v^3}6+\frac v2 \Big)   
$$
the coefficients of the first and the second fundamental forms are respectively
$$
	E= \frac14(1-u^2+v^2)^2 \qquad F=0  \qquad G= - \frac14(1-u^2+v^2)^2 \ ,
$$
$$
	L=0  \qquad M=1  \qquad  N=0 \ .
$$
Then for the Gauss curvature and the mean curvature we obtain
$$
	K=\frac{16}{(1-u^2+v^2)^4}>0 \qquad  H=0   \qquad K-H^2=\frac{16}{(1-u^2+v^2)^4}>0  \ .
$$
Hence the surface admits real asymptotic lines. Moreover, it can be checked that  
$(u,v)$ are canonical asymptotic parameters. 

This surface is an Enneper-type time-like surface of positive Gauss curvature.

\end{example}

\begin{example} \label{ex:3}   Consider the rotational surface given by
$$
	z(u,v)=\Big( u,\cos u\cosh v,\cos  u\sinh v \Big) 
$$
for $|u|<\pi/2$. Then the coefficients of the first and the second fundamental forms are respectively
$$
	E=1+\sin^2u  \qquad F=0  \qquad G=-\cos^2u \ ,
$$
$$
	L=\frac{\cos u}{\sqrt{1+\sin^2u}}  \qquad M=0  \qquad  N=-\frac{\cos u}{\sqrt{1+\sin^2u}}  \ .
$$
Hence for the Gauss and the mean curvature we obtain
$$
	K=\frac{1}{(1+\sin^2 u)^2}   \qquad H=\frac1{\cos u\big(\sqrt{1+\sin^2u}\big)^3}  \qquad K-H^2=-\frac{\sin^2u\tan^2u}{(1+\sin^2u)^3} \ .
$$
Here $K-H^2<0$, but $K>0$, so the surface admits both principal and asymptotic parameters. For example after the
change 
$$
	u=\sinh\bar u+\bar v  \qquad  v=\sinh\bar u-\bar v
$$
we see that the parameters $(\bar u, \bar v)$ are asymptotic. But  in these parameters 
$\overline E>0$, $\overline G>0$, so we can not use our constructions.
\end{example}

\begin{example} \label{ex:4}  Consider  the Lorentz sphere with parametrization
$$
	z(u,v)=\left( \frac{\cosh v}{\cosh u},\tanh u,\frac{\sinh v}{\cosh u} \right)  \ .
$$
Then 
$$
	E=\frac1{\cosh^2u }   \qquad   F=0   \qquad   G=-\frac1{\cosh^2u }  \ ,
$$
$$
	L=-\frac1{\cosh^2u}    \qquad   M=0   \qquad  N=\frac1{\cosh^2u}  \ .
$$
Consequently, $K=1$, $K-H^2=1$. Trying to obtain asymptotic parameters we put
$$
	u=\bar u-\bar v    \qquad    v=\bar u +\bar v \ .
$$ 
Now we find that
the resulting parameters are isotropic, i.e. $\bar E=\bar G=0$. So our method is not applicable.

\end{example}

The last two examples show that 
the requirement $K-H^2>0$ is essential for our method.

{\bf Funding.} The author is partially supported by the National Science Fund, Ministry of Education and Science of Bulgaria, 
under contract KP-06-N82/6.

\end{document}